\documentclass[12pt]{amsart}
\usepackage{amssymb,amsmath,amsthm,mathtools,amscd,mathrsfs,enumitem,nccmath}
\usepackage[margin=1in]{geometry}
\usepackage{mleftright}
\mleftright

\def\A{\mathcal{A}}
\def\B{\mathcal{B}}
\def\Z{\mathbb{Z}}
\def\Q{\mathbb{Q}}
\def\R{\mathbb{R}}

\def\a{\alpha}
\def\b{\beta}
\def\g{\gamma}
\def\d{\delta}
\def\l{\lambda}
\def\L{\Lambda}
\def\half{\tfrac{1}{2}}

\def\GL{{\rm GL}}
\newcommand*{\defeq}{\stackrel{\text{def}}{=}}

\newcommand{\pMatrix}[4]{\left(\begin{matrix}#1 & #2 \\ #3 & #4\end{matrix}\right)}
\renewcommand{\pmatrix}[4]{\left(\begin{smallmatrix}#1 & #2 \\ #3 & #4\end{smallmatrix}\right)}

\def \ep {\varepsilon}

\newtheorem*{theorem1*}{Dirichlet Approximation Theorem}
\newtheorem*{theorem2*}{Minkowski Approximation Theorem}
\newtheorem*{theorem3*}{Minkowski's First Convex Body Theorem}
\newtheorem*{theorem*}{Theorem}
\newtheorem{theorem}{Theorem}
\newtheorem{lemma}{Lemma}
\newtheorem{corollary}{Corollary}
\newtheorem*{corollary*}{Corollary}

\theoremstyle{remark}

\newtheorem*{remark}{Remark}
\newtheorem*{remarks}{Remarks}
\newtheorem*{examples}{Examples}

\numberwithin{equation}{section}
\numberwithin{lemma}{section}
\title{The Minkowski chain and Diophantine approximation}
\date{\today}
\author{Nickolas Andersen}
\address{Brigham Young University,
Department Of Mathematics, Provo, UT 84602} \email{nick@math.byu.edu}

\author{William Duke}
\address{UCLA Mathematics Department,
Box 951555, Los Angeles, CA 90095-1555} \email{wdduke@ucla.edu}
\thanks{Supported by NSF grant DMS 1701638.}

\begin{document}

\begin{abstract}
The  Hurwitz chain  gives a sequence of pairs of Farey approximations to an irrational real number.  Minkowski gave  a criterion for a number to be algebraic by using a certain generalization of the Hurwitz chain.
 We  apply Minkowski's generalization (the Minkowski chain)  to give  criteria for 
a real linear form  to be either badly approximable or singular. 
We also give a variant of Dirichlet's approximation theorem for a real linear form that 
produces a whole basis of approximating integral vectors rather than a single one.
This result  holds if and only if the form is badly approximable. 
The proofs rely on properties of successive minima and reduced bases  of lattices.
\end{abstract}

\maketitle

\section{Introduction }\label{intro}
Every irrational $\a\in \R$ has a unique expansion as an infinite regular continued fraction
\begin{equation*}
\a=a_0+\frac{1}{a_1+}\;\frac{1}{a_2+}\;\frac{1}{a_3+}\cdots
\end{equation*}
where $a_j$ are integers called the partial quotients of $\a$ with  $a_j>0$ for $j \geq 1.$
A striking result of elementary number theory, going back to Euler and Lagrange,  is that $\a$ is algebraic of degree two over $\Q$  if and only if this expansion is eventually periodic. 

More generally,  suppose that $\a\in \R$   is such that  $\{\a^n,\a^{n-1},\dots,\a,1\}$ are  linearly independent over $\Q$.
The $n=1$ case above leads naturally to the following problem.  Find an algorithm, like the regular continued fraction,  which
provides a criterion for $\a$  to be algebraic of degree $\ell=n+1$ over $\Q$.
Since   Jacobi \cite{Jac}, most investigations of multi-dimensional generalizations of continued fractions,
as applied to algebraic numbers,  have concentrated on periodicity.
This approach has had only limited success. 

However, already in 1899 Minkowski \cite{Mink1.2}\footnote{A translation (with additions) of this paper into English is given in Vol.~1 Chap.~IX of \cite{Han}.} found such an algorithm that produces a sequence of nonsingular $\ell \times \ell $ integral matrices, the  Minkowski chain, which characterizes algebraic  $\a$  not through periodicity 
but rather  a certain finiteness condition.
The Minkowski chain generalizes the Hurwitz chain, itself a refinement of  the regular continued fraction.
 In a speech appearing as the preface to Minkowski's collected papers,\footnote{See p.~XV.~of Vol.~I of the Gesammelte Abhandlungen.} Hilbert said  that ``Der Minkowskische Algorithmus ist nicht ganz einfach...." One goal of our paper is to  revive interest in the Minkowski chain and its applications. 
In particular, Minkowsi's criterion for an algebraic number  has not received  the attention we think it deserves.    

Our main goal is to apply the Minkowski chain to characterize badly approximable and singular real linear forms in several variables. 
We also give a variant of Dirichlet's approximation theorem for a linear form that produces a whole basis of approximating integral vectors
and that holds precisely for badly approximable forms.

In the next section we recall the definitions of the Hurwitz and Minkowski chains, formulate their relationships to each other and to the regular continued fraction  and state Minkowski's criterion. We also give some illustrative examples.
Then in \S \ref{appl} we state our results on Diophantine approximations by linear forms.
The remainder of the paper contains the proofs.
We  have tried to make the presentation as self-contained as is feasible and we provide proofs of all numbered theorems, corollaries  and lemmas.

\section{The Minkowski chain}\label{minko}
Suppose that $\a\in (0,1)$ is irrational. A natural way to approximate $\a$ by rational numbers, while controlling the size of the of the denominators,  is to use Farey fractions. For $m\in \Z^+$ let $\mathcal{F}_m$ be the $m^{th}$ Farey set, which consists of all rational numbers in $[0,1]$ in increasing order  whose denominators are at most $m$. Thus
\[
\mathcal{F}_1=\{\tfrac{0}{1},\tfrac{1}{1}\},\,\mathcal{F}_2=\{\tfrac{0}{1},\tfrac{1}{2},\tfrac{1}{1}\},\;\mathcal{F}_3=\{\tfrac{0}{1},\tfrac{1}{3},\tfrac{1}{2},\tfrac{2}{3},\tfrac{1}{1}\}, \;\mathcal{F}_4=\{\tfrac{0}{1},\tfrac{1}{4},\tfrac{1}{3},\tfrac{1}{2},\tfrac{2}{3},\tfrac{3}{4},\tfrac{1}{1}\},\dots.
\]
 For a fixed  $m$ let $(\frac{p}{q},\frac{p'}{q'}) $ be the unique pair of successive Farey fractions in $\mathcal{F}_m$ with  $\frac{p}{q}<\a<\frac{p'}{q'}$. 
  After $m=2$ the pair of surrounding fractions might not change as $m$ increases to $m+1$, 
but when it does one fraction will remain and the new one will be $\frac{p+p'}{q+q'}$.
This process was studied in some detail by Hurwitz \cite{Hur} in 1894
and the sequence of (distinct) Farey pairs is called the Hurwitz chain for $\a$ by Philippon in \cite{Phi}.

We can encode the Hurwitz chain of an irrational $\a\in (0,1)$ by a unique infinite word in the letters $R$ and $L$.   We label a pair with $R$ if within the pair the old fraction is to the right of the new one and $L$ if it is to the left. We label the first pair  $(\tfrac{0}{1},\;\tfrac{1}{1})$ with $L$ and the next with $R$ if it  is $ (\tfrac{1}{2},\tfrac{1}{1})$ and with $L$ if it is $(\tfrac{0}{1},\tfrac{1}{2})$.

For example,  the Hurwitz chain for $\a=\frac{1}{2}(-1+\sqrt{5})$  begins
 \begin{equation}\label{hc}
(\tfrac{0}{1},\;\tfrac{1}{1}),\; (\tfrac{1}{2},\tfrac{1}{1}),\;  (\tfrac{1}{2},\tfrac{2}{3}),\; (\tfrac{3}{5},\tfrac{2}{3}),\; (\tfrac{3}{5},\tfrac{5}{8}), \;(\tfrac{8}{13},\tfrac{5}{8}),\dots
 \end{equation}
with corresponding word $LRLRLR\dots.$ 

The word corresponding to 
the Hurwitz chain for $\a\in (0,1)$ determines the partial quotients $a_j$ in 
 the regular continued fraction
\begin{equation}\label{scf2}
\a=\frac{1}{a_1+}\;\frac{1}{a_2+}\;\frac{1}{a_3+}\cdots.
\end{equation}
It follows from standard properties of the convergents of the continued fraction that
 $a_j$  is given by the number of successive $L$'s or $R$'s in the $j^{th}$ block of the word. 
 Thus the partial quotients for  $\a=\frac{1}{2}(-1+\sqrt{5})$ are $a_j=1$ for all $j$.
Clearly $\a$ is quadratic over $\Q$   if and only if  the word associated to the Hurwitz chain for $\a$  is eventually periodic.

Minkowski discovered that to detect algebraic numbers of degree greater than two it is better to abandon  periodicity. His algorithm is readily described.  We give it in a slightly generalized form that we need later.   Suppose that $(\a_1,\a_2,\dots,\a_n)\in \R^n$
is such that $\{\a_1,\dots,\a_n,1\}$
 are linearly independent over $\Q$.
Set \[\ell=n+1.\] Define for any real matrix $A=(a_{i,j})$ the norm $\|A\|_\infty=\max(|a_{i,j}|).$
For $\ell\geq 2$ and
 $m\in \Z^+$ let $\mathcal{A}_m$ consist of all non-singular integral $\ell\times \ell$ matrices $A$ with $\|A\|_\infty\leq m.$
Write
\begin{equation}\label{be}
A(\a_1,\dots,\a_n,1)^\top=(\b_1,\b_2,\dots,\b_{\ell})^\top.
\end{equation}
Let  $\A_{m,1}\subset \mathcal{A}_m$ be those $A\in \A_m$ that minimize  $\|A(\a_1,\dots,\a_n,1)^\top\|_\infty$
and for which the minimum is $|\b_1|$.
 This fixes the first row of $A$ by the linear independence assumption, provided we make some sign convention, for example  that the first non-zero entry in the first row is positive.
 Next let $\A_{m,2}\subset \mathcal{A}_{m,1}$ be those $A\in \mathcal{A}_{m,1}$ for which $|\b_2|$ gives the minimal value  thereby with the corresponding convention fixing the second row of $A$. 
Continue this process of defining rows of $A$.
Thus for each $m$ we have defined $A_m$ uniquely.
The matrices $A_m$ need not change as $m$ goes to $m+1$.
Let $B_k=A_{m_k}$, where $k=1,2,\dots$,  define the subsequence of distinct matrices starting with $B_1=A_1.$
The sequence $\{B_1,B_2,\dots\}$ of matrices is what we will call the {\it Minkowski chain} for $(\a_1,\dots,\a_n)$.
 
When $n=1$ and $\a\in (0,1)$  the Minkowski chain corresponds to the Hurwitz chain for $\a$.
More precisely, we have the following result.
\begin{theorem}\label{mchc}
Let the $k^{th}$ matrix in the Minkowski chain for an irrational $\a\in (0,1)$ be 
\[
B_k=\pMatrix{q}{-p}{q'}{-p'}.
\]
Then the $k^{th}$ pair  in the Hurwitz chain
 is either  $(\frac{p}{q},\frac{p'}{q'}) $ or $(\frac{p'}{q'},\frac{p}{q}) $.
\end{theorem}
An immediate corollary is the following fact which, as far as we know, need not hold in general for $n>1$.
\begin{corollary}
When $n=1$ we have that $|\!\det B_k|=1$ for all $k$.
\end{corollary}

Let any $\ell\times \ell $ matrix $B=(b_{i,j})$ act on an $n$-tuple $(x_1,x_2,\dots,x_n)$ projectively as a linear fractional map:
\[
B(x_1,\dots, x_n)=\Big(\tfrac{\sum_{j=1}^{n} b_{1,j} x_j+b_{1,\ell}}{\sum_{j=1}^{n} b_{\ell,j} x_j+b_{\ell,\ell}}, \dots,\tfrac{\sum_{j=1}^{n} b_{n,j} x_j+b_{n,\ell}}{\sum_{j=1}^{n} b_{\ell,j} x_j+b_{\ell,\ell}}\Big).
\]
For each $k\in \Z^+$ set
\begin{equation}\label{alph}
B_k(\a_1,\dots,\a_n)=(\a_{k,1},\dots,\a_{k,n}),
\end{equation}
where $B_k$ is the $k^{th}$ matrix in the Minkowski chain for $(\a_1,\dots,\a_n)$.
Clearly we have that \[0<|\a_{k,1}|<|\a_{k,2}|<\cdots<|\a_{k,n}|<1.\]
We are mostly interested in properties of the sequence $\{(\a_{k,1},\dots,\a_{k,n})\}_{k\geq 1}$ of $n$-tuples
attached to  $(\a_1,\dots,\a_n).$
 Minkowski realized that it is the finiteness of the set of  $n$-tuples $B_k(\a^n,\dots,\a)$, rather than periodicity determined by the chain, which  characterizes algebraic $\a$ of degree $\ell.$

\begin{theorem}[Minkowski \cite{Mink1.2}]\label{mac}
Suppose that $\a\in \R$ and that $\{\a^n,\a^{n-1},\dots,\a,1\}$ are linearly independent over $\Q$. 
 Then $\a$ is algebraic of degree $\ell=n+1$ over $\Q$ 
   if and only if 
the sequence $\{B_k(\a^n,\a^{n-1},\dots,\a)\}_{k\geq 1}$ contains only finitely many different $n$-tuples. 
\end{theorem}

Actually,  Minkowski's formulation   allows $\a$ to be complex. He also did not assume that $\{\a^n,\dots,\a,1\}$ are linearly independent
over $\Q$, but by using the algorithm with smaller $n$ we may assume this without any loss and with uniqueness of the expansion.

\begin{examples} \hfill
\begin{enumerate}[label=(\roman*)]
	\item The Minkowski chain for $\a=\frac{-1+\sqrt{5}}{2}$ is
\begin{equation*}
	B_1= \pmatrix 1{-1}1{0}, \;\;  B_2=\pmatrix 2{-1}1{-1}, \;\; B_3=\pmatrix 3{-2}2{-1},  \dots, B_k=\pmatrix {F_{k+1}}{-F_k}{F_k}{-F_{k-2}},\dots,
	\end{equation*}
	which corresponds to (\ref{hc}). Here $F_k$ is the $k^{th}$ Fibonnaci number and for each   $k$
	 \[B_k(\tfrac{-1+\sqrt{5}}{2})=\tfrac{1-\sqrt{5}}{2}.\] 
	\item Let $\theta=2\cos\left(\frac{2\pi}{7}\right)$ so that $\Q(\theta)$ is the real cubic field of discriminant $49$, i.e.~the splitting field of $x^3+x^2-2x-1$.
The Minkowski chain for $(\theta^2,\theta)$ begins
\begin{equation*}
B_1 = \left(
\begin{smallmatrix}
 0 & 1 & -1 \\
 1 & -1 & 0 \\
 1 & -1 & -1 \\
\end{smallmatrix}
\right), \quad
B_2 = \left(
\begin{smallmatrix}
 1 & -2 & 1 \\
 2 & -1 & -2 \\
 0 & 1 & -1 \\
\end{smallmatrix}
\right), \quad
B_3 = \left(
\begin{smallmatrix}
 1 & -2 & 1 \\
 3 & -3 & -1 \\
 2 & 0 & -3 \\
\end{smallmatrix}
\right), \quad
B_4 = \left(
\begin{smallmatrix}
 1 & 2 & -4 \\
 1 & -2 & 1 \\
 3 & -3 & -1 \\
\end{smallmatrix}
\right), \ldots.
\end{equation*}
By Theorem \ref{mac} we know that the set of values $\{B_k(\theta^2,\theta)\}$ is finite.
Among the first 30 terms there are only six distinct pairs up to sign, namely
\newcommand\sdots{\makebox[1em][c]{.\hfil.\hfil.}}
\begin{gather*}
(0.15883\sdots, 0.64310\sdots), \quad (0.24698\sdots, 0.55496\sdots), \quad (0.35690\sdots, 0.44504\sdots), \\ 
(0.44504\sdots, 0.80194\sdots), \quad (0.55496\sdots, 0.69202\sdots), \quad (0.64310\sdots, 0.80194\sdots).
\end{gather*}
	\item Suppose that $\a$ is transcendental, so $\{\a^n,\dots,\a,1\}$ are linearly independent over $\Q$ for any positive integer $n$. For a fixed $n$ let  $B_k(\a^n,\dots,\a)=(\a_{k,1},\dots,\a_{k,n})$ come from the Minkowski chain for $(\a^n,\dots,\a)$ as above.
By Theorem \ref{mac} we know that \[\{(\a_{k,1},\dots,\a_{k,n})\}_{k\geq 1}\]   is an infinite set.

	Recall that $\a\in \R$ is a Liouville number if, for every positive integer $m$, there exist infinitely many relatively prime integers $p,q$ with $q>0$ such that
\[0<|\a-\tfrac{p}{q}|<q^{-m}.\] Liouville's theorem on Diophantine approximation implies that  a Liouville number $\a$ is transcendental.
If $\a$ is a Liouville number and $n \in \Z^+$ is fixed, our results stated below imply that  not only is $\{(\a_{k,1},\dots,\a_{k,n})\}_{k\geq 1}$  infinite, but also $ |\a_{k,1}|$ gets arbitrarily close to zero as $k \rightarrow \infty$.
For the Liouville constant
\begin{equation*}
\l=\sum_{m\geq 1} 10^{-m!}= 0.11000100000000000000000100\dots
\end{equation*}
and the cases $n=1,2,3$, the behavior of $|\l_{k,1}|$  is shown in Figure \ref{fig1}.

\begin{figure}
	\includegraphics[scale=0.35]{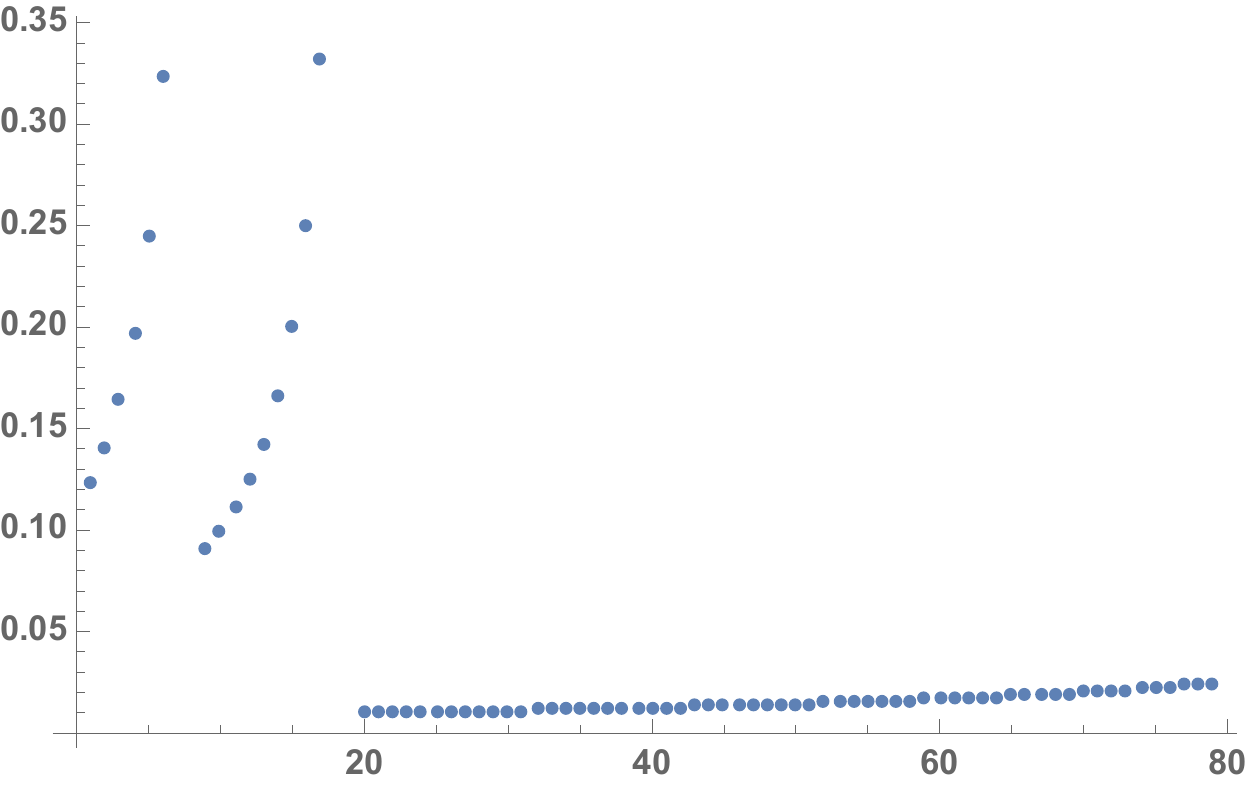}	\quad \includegraphics[scale=0.35]{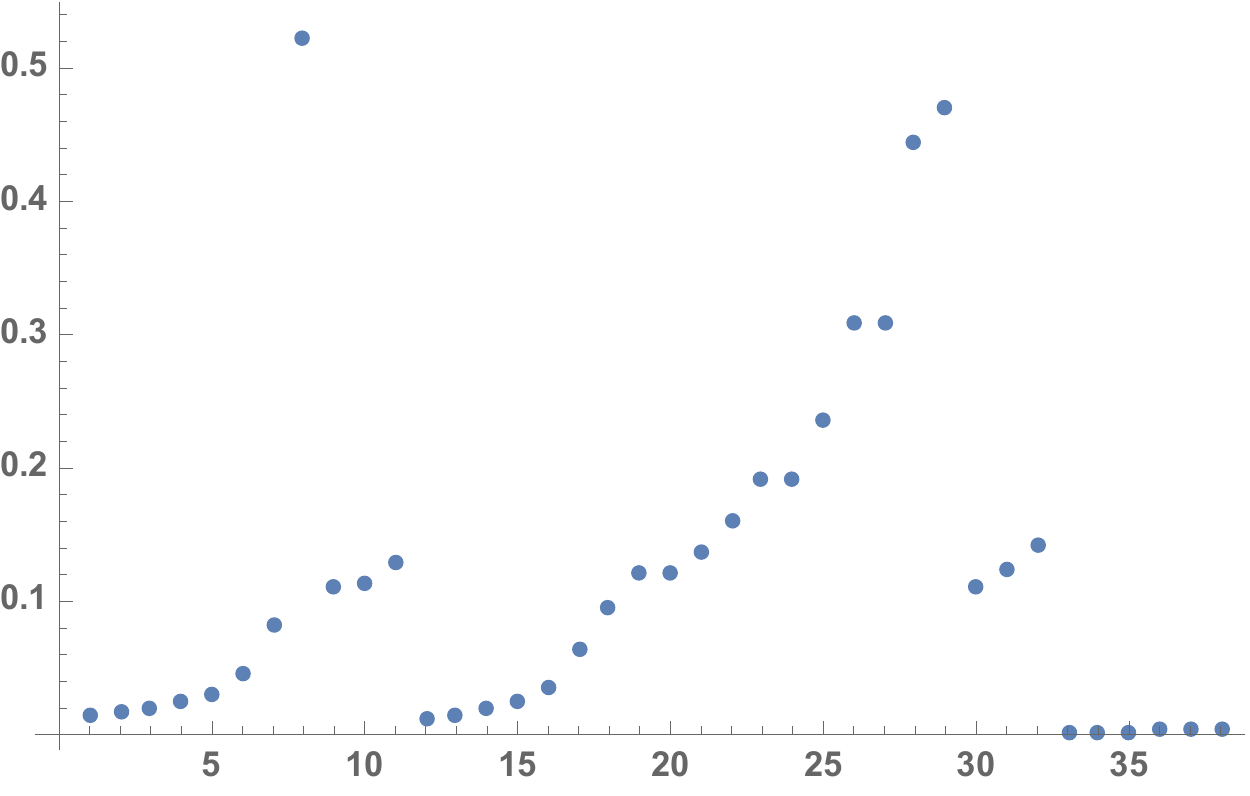}\quad \includegraphics[scale=0.35]{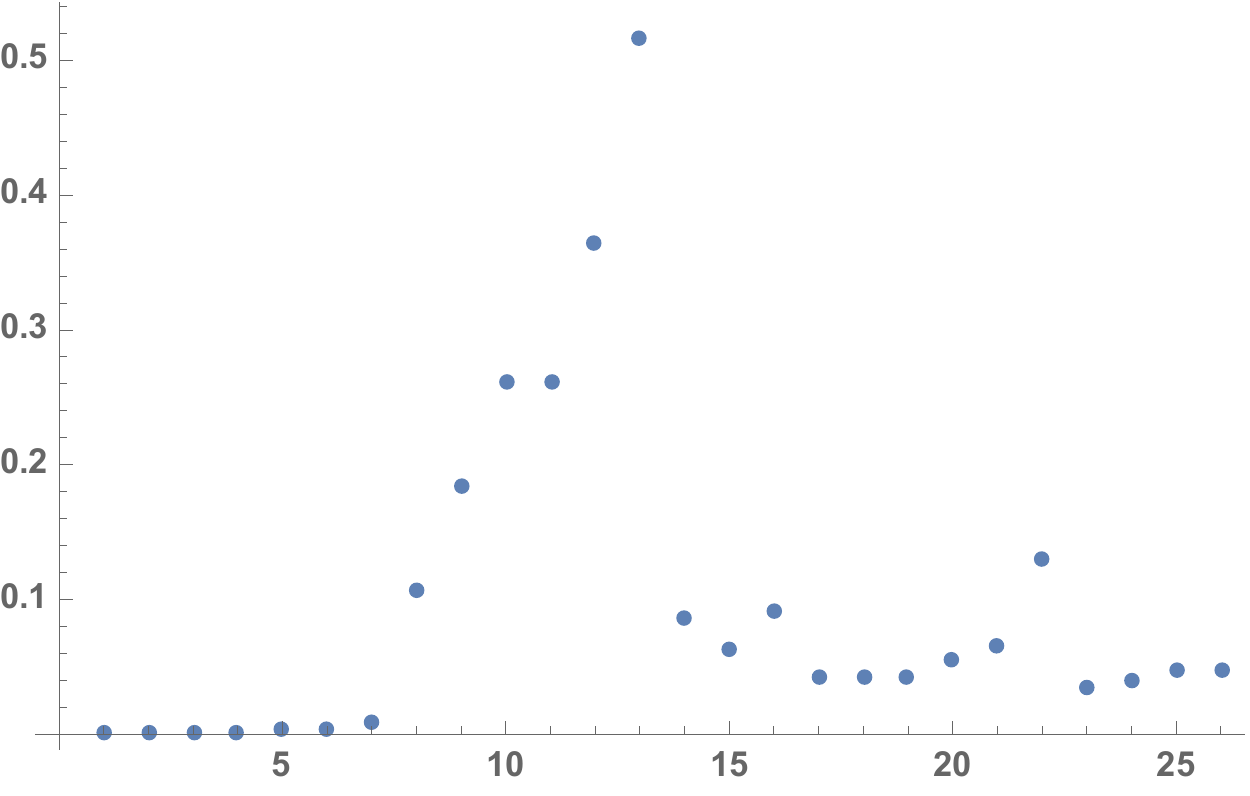}
	\caption{The sequences $|\l_{k,1}|$ for $n=1,2,3$. }
	\label{fig1}
\end{figure}
\end{enumerate}
\end{examples}

\section{Applications to Diophantine approximation}\label{appl}
Our main  object is to apply the Minkowski chain and related methods  to certain problems on Diophantine approximation by badly approximable and by singular real linear forms in two or more variables.

Associate to any $\a=(\a_1,\dots,\a_n)\in \R^n$ the  linear form\footnote{Our abuse of notation in using $\a$ as an $n$-tuple of real numbers  and as a number, depending on the context, is convenient and should not cause confusion.} \[L_\a(x)=\a_1x_1+\cdots +\a_nx_n.\]
Those  $\{\a_1,\dots,\a_n,1\}$ that give a basis over $\Q$ for a real number field
have the following well-known Diophantine approximation property  \cite[Thm 4A p.~42]{Sch3}. There is a constant $c=c_\a>0$ so that for any non-zero $q=(q_1,\dots,q_n) \in \Z^n$  
 \begin{equation}\label{ba}
\|L_\a(q)\|\geq c\|q\|_\infty^{-n}.
 \end{equation}
Here $\|t\|$ denotes the distance from a real $t$ to the nearest integer. For any $\a\in \R^n$ 
if the form $L_\a(x)$  satisfies (\ref{ba}) then $L_\a$ is said to be {\it badly approximable}.  
For simplicity we shall also sometimes say that $\a$ is badly approximable.
 It is known that the set of all badly approximable $\a\in \R^n$  has Lebesgue measure zero \cite{Khi} yet has full Hausdorff dimension $n$, hence includes $\a$ for which $\{\a_1,\dots,\a_n,1\}$ does not give a $\Q$-basis for a number field  \cite{Sch1}.  For more on the history of these results see  \cite{Sch3} and its references.
 
 A natural problem presents itself; can we formulate a  criterion for a form in $n$ variables  to be badly approximable 
 using the Minkowski chain?  
\begin{theorem}\label{t3} Suppose that 
  $\{\a_1,\dots,\a_n,1\}$ are linearly independent over $\Q$ and that $\a_{k,1}$  is given by the Minkowski chain for $\a$.  Then
the form $L_\a$  is badly approximable
if and only if $ |\a_{k,1}|$  is bounded away from 0.
\end{theorem}

Theorem \ref{t3} generalizes a well-known characterization of badly approximable numbers in case $n=1.$
\begin{corollary}\label{cor}
An irrational $\a\in \R $ is badly approximable if and only if the partial quotients in its regular continued fraction expansion 
 are bounded.
\end{corollary}
This follows from Theorem \ref{t3} and Lemma \ref{pqe} proved below.
 For the standard direct  proof see  \cite[Thm 5F p.~22]{Sch3}.

 An important property of any badly approximable $L_\a$, discovered  by Davenport and Schmidt,
is that Dirichlet's approximation theorem  can be improved for it in the following sense.
\begin{theorem*}[Davenport-Schmidt \cite{DS1}]
If $L_\a$ is badly approximable then
there exists a constant $c<1$ having the following property:
for every sufficiently large integer $Q$   there are non-zero $q\in \Z^n$  with
$\|q\|_\infty\leq Q$ for which
\begin{equation}\label{dir1} 
\| L_\a(q)\|\leq c\, Q^{-n}.
\end{equation}
\end{theorem*}
Of course, the improvement is in the fact that $c<1$ rather than $c=1.$
The proof of this result relies on a conjecture of Minkowski, proven by Haj\'os \cite{Haj}, which says that in any lattice tiling of space by cubes, there are two cubes  meeting face to face.

  The theorem of Davenport and Schmidt implies the following  result.
\begin{theorem}\label{liou}
Suppose that $\a$ is a Liouville number.  Then $L_{(\a^n,\dots,\a)}$ is not badly approximable for any $n.$
\end{theorem}
Note that this gives a strengthening of Liouville's result that a Liouville number is transcendental.
The  claim  that for a Liouville number $ |\a_{k,1}|$ gets arbitrarily close to zero as $k \rightarrow \infty$, which was made  in  example (iii) from the previous section, is a  consequence of Theorems \ref{liou} and \ref{t3}.

When $n=1$ every irrational $\a$  for which Dirichlet's  theorem can be improved is badly approximable.
However,  for $n>1$ there exist $L_\a$ with $\{\a_1,\dots,\a_n,1\}$ linearly independent over $\Q$
that are not badly approximable but for which Dirichlet's theorem can be improved.  
In fact, Dirichlet's theorem can sometimes be ``infinitely improved." More precisely
say $L_\a$ (or $\a$) is  {\it singular} if
for any $\ep>0$ there is a $Q_\ep$ so that if $Q\geq Q_\ep$ there is a $q\in \Z^n$ with \[0<\|q\|_\infty \leq Q\;\;\text{ such that}\;\;\|L_\a(q)\| \leq \ep\,Q^{-n}.\]
Such forms  are clearly not badly approximable.
Starting with work of Khinchine \cite{Khi2}  it is known that singular $L_\a$ with $\{\a_1,\dots,\a_n,1\}$ linearly independent over $\Q$  exist if $n>1$, although apparently  no explicit example has been found.
It has recently been shown  that when $n>1$  the set of singular $\a\in \R^n$ has Hausdorff dimension $\frac{n^2}{n+1}$
(see \cite{Ch}, \cite{CC}, and \cite{DFSU1}). 
The Minkowski chain  also gives a criterion for $L_\a$ to be singular.
 
 \begin{theorem}\label{t4} Suppose that 
  $\{\a_1,\dots,\a_n,1\}$ are linearly independent over $\Q$.  Then
the form $L_\a$  is singular
if and only if $ |\a_{k,1}|\rightarrow 0 $  as $k \rightarrow \infty$.
\end{theorem}

Our final result gives  a different kind of improvement (in some sense) of Dirichlet's theorem
that holds  for badly approximable forms and only those. This theorem gives a whole basis of integral vectors, rather than just one, 
for which a bound  of Dirichlet's type holds (with $c>1$ allowed).

 \begin{theorem}\label{t1}
The form $L_\a$ where 
 $\a=(\a_1,\dots,\a_n)\in \R^n$ is badly approximable if and only if there exists a constant $c>0$ so that for any $Q\in \Z^+$ there is an $A\in \GL(\ell,\Z)$ with $\|A\|_\infty < Q$
and
\[
\| A(\a_1,\dots,\a_n,1)^\top\|_\infty < c \,Q^{-n}.
\]
\end{theorem}

This has as an immediate consequence the following result.
\begin{corollary}
Suppose that $\{\a_1,\dots,\a_n,1\}$ are linearly independent over $\Q$.
If $L_\a$ is badly approximable then there is a $c>0$ so that there are infinitely many $A\in \GL(\ell,\Z)$ with
\[
\|A(\a_1,\dots,\a_n,1)^\top\|_\infty < c \|A\|_\infty^{-n}.
\]
\end{corollary}

\begin{remarks} \hfill
\begin{enumerate}[label=(\roman*)] 
\item Theorems \ref{t3} and \ref{t4} differ  substantially from the dynamical criteria for bad approximability and singularity given (more generally for systems of forms) by Dani \cite{Da}. Roughly speaking, he showed that badly approximable systems of forms correspond to certain bounded trajectories in the space of unimodular lattices while singular systems correspond to 
divergent trajectories. In fact,  a  version of  these criteria in the case of a single form  is one ingredient in our proofs of Theorems \ref{t3} and \ref{t4} (see Lemma \ref{l1}).

\item It would be interesting to find generalizations of Theorems \ref{mac} through \ref{t1}  that apply to systems of forms.
As we mentioned, Minkowski already obtained Theorem \ref{mac} for certain forms with complex coefficients
and so it is natural to consider generalizations of the other results for them as well.

\item In all cases that we have checked numerically, each matrix in the Minkowski chain has been in $\GL(\ell,\Z)$. 
 \end{enumerate}
 \end{remarks}
 
In the next section we prove Theorem  \ref{mchc}. 
Then in \S\ref{mp} we prove  Theorem \ref{mac}.  We  provide in \S \ref{redu} background results in reduction theory. In \S \ref{pf} we give the proofs of Theorems \ref{t3},  \ref{liou}  and \ref{t4} and we prove Theorem \ref{t1}  in \S \ref{di}. 

\section{The Hurwitz chain}\label{add}
In this section we prove Theorem  \ref{mchc} and a lemma relating the partial quotients of $\a\in (0,1)$
to the Minkowski chain of $\a$
when $n=1$.
We require an elementary lemma about Farey fractions.
We always assume that rational fractions  are in lowest form.
  \begin{lemma}\label{lf}
Suppose that $\frac{p}{q}<\frac{p'}{q'} $ is a pair of successive fractions in $\mathcal{F}_m$ and that $\a\in (\frac{p}{q},\frac{p'}{q'}) $ is irrational.
 Then
  \begin{enumerate}[label=(\roman*)] 
\item $ |q\a-p|<|q'\a-p'| $ if and only if $ \a\in (\tfrac{p}{q},\tfrac{p+p'}{q+q'})$

\item The fraction  $\frac{p+p'}{q+q'}$ is the unique fraction with the smallest denominator greater than $m$ that is closer to 
$\a$ than at least one of $\frac{p}{q},\frac{p'}{q'}$. 

\end{enumerate}
  \end{lemma}
 \begin{proof}
(i)  If $\a\in (\frac{p}{q},\frac{p+p'}{q+q'})$ then $|\a-\frac{p'}{q'}|>|\frac{p+p'}{q+q'}-\frac{p'}{q'}|=\frac{1}{q'(q+q')}$
 so $|q'\a-p'|>\frac{1}{(q+q')}$.  Similarly  $|q\a-p|<\frac{1}{(q+q')}$ so $|q\a-p|<|q'\a-p'| $ in this case.
 The converse is similar using $\a\in (\frac{p+p'}{q+q'},\frac{p'}{q'})$.
 
 \bigskip
 
(ii) It is well-known (see e.g. \cite[p. 4]{Sch3}) that
 $\frac{p+p'}{q+q'}$ is the unique fraction with the smallest denominator greater than $m$ that is between $\frac{p}{q}$ and $\frac{p'}{q'}.$
Thus we need only show that $\frac{p+p'}{q+q'}$  is closer to $\a$ than any other $\frac{p''}{q''}$ with $m<q''\leq q+q'$
 and either $\frac{p''}{q''}<\frac{p}{q}$  or $\frac{p''}{q''}>\frac{p'}{q'}$. 
 
 Suppose that  $\frac{p''}{q''}<\frac{p}{q}$. If $\a>\tfrac{p+p'}{q+q'}$ we are done so assume that 
 \begin{equation}\label{mis}
 \tfrac{p}{q}<\a<\tfrac{p+p'}{q+q'}.
 \end{equation}
 Now
 \[
|\a- \tfrac{p''}{q''}|>|\tfrac{p''}{q''}-\tfrac{p}{q}| \geq \tfrac{1}{q''q} \geq  \tfrac{1}{q(q+q')} 
 \]
 while by (\ref{mis})
 \[
 |\a- \tfrac{p+p'}{q+q'}|<|\tfrac{p+p'}{q+q'}-\tfrac{p}{q}| =\tfrac{1}{q(q+q')}.
 \]
 
 The case $\frac{p''}{q''}>\frac{p'}{q'}$ is similar.
 \end{proof}
 
\subsection*{Proof of Theorem \ref{mchc}}
We want to show that if
\[
B_k=\pMatrix{q_k}{-p_k}{q_k'}{-p_k'}
\]
then the $k^{th}$ pair  in the Hurwitz chain
 is either  $(\frac{p_k}{q_k},\frac{p_k'}{q_k'}) $ or $(\frac{p_k'}{q'_k},\frac{p_k}{q_k}) $.
 
This follows by induction on $k$. It holds for $k=1.$
   Suppose it holds for some $k\geq 1.$ Thus $\frac{p_k}{q_k},\frac{p_k'}{q_k'}$ or $\frac{p'_k}{q'_k},\frac{p_k}{q_k}$ are successive Farey fractions
   in $\mathcal{F}_m$ where $m= \max(q_k,q'_k).$
   
By the definition of the Minkowski chain given around (\ref{be}) we know that 
   the first row of $B_k$   (the one with $|q_k\a-p_k|$ minimal) must appear in $B_{k+1}$ as either the first row or the second row.
  Now (i) of Lemma \ref{lf} implies that the fraction associated to the row retained is the one retained by the Hurwitz chain.
   
   Thus we must show that the new row of $B_{k+1}$, say $(q'',-p'')$, is precisely $(q_k+q_k',-(p_k+p_k'))$.
 By the definition of the Minkowski chain  $q''> m$ and certainly
 \[
 |q''\a-p''|< |q'_k\a -p_k'|.
 \]
 Thus $|\a-\frac{p''}{q''}|<|\a-\frac{p'_k}{q'_k}|$ so by (ii) of Lemma \ref{lf} we know that $q''\geq q_k+q_k'.$   Now 
  \begin{equation}\label{sum}
 |( q_k+q_k')\a-(p_k+p'_k)|=|(q_k\a-p_k)+(q'_k\a-p_k')|.
  \end{equation}
  Also, $\a-\frac{p_k}{q_k}$ and  $\a-\frac{p'_k}{q'_k}$ have different signs 
  hence so do $q_k\a-p_k$ and  $q_k'\a-p_k'$.
  By construction of $B_k$ we know that $|q_k\a-p_k|<|q_k'\a-p_k'|.$
   Therefore by (\ref{sum}) we have that
   \[
   |( q_k+q_k')\a-(p_k+p'_k)|<|q'_k\a -p_k'|.
   \]
   It follows that $(q'',-p'')=(q_k+q_k',-(p_k+p_k'))$.
   This completes the proof of Theorem \ref{mchc}.
   \qed
   
\bigskip
It is easy to give a formula for the 
$k^{th}$ pair in the Hurwitz chain for $\a\in (0,1)$ in terms of the partial quotients $a_j$  of $\a$.
For a fixed $k \in \Z^+$ write $k=a_1+\cdots +a_j+a$ where $0\leq  a<  a_{j+1}$.
Set $R=\pmatrix{1}{1}{0}{1}$ and $L=\pmatrix{1}{0}{1}{1}$ and let $A=L$ if $j$ is even and $A=R$ if $ j$ is odd. Then the $k^{th}$ pair in the Hurwitz chain for $\a\in (0,1)$ is given by $(\frac{p_k}{q_k},\frac{p_k'}{q_k'})$,
where 
\begin{equation}\label{trans}
\pmatrix{p_k'}{p_k}{q'_k}{q_k}=L^{a_1}R^{a_2}\cdots  A^{a}.
\end{equation}

Let $b=a$ if $a>0$ and $b=a_j$ otherwise.
Then by Theorem \ref{mchc} 
we have for $B_k$ from the Minkowski chain the formula
 $B_k = MB_{k-b}$, where $M$ is either $L^b$, $R^b$, $\pmatrix 0110L^b$, or $\pmatrix 0110R^b$.

The following consequence of these formulas is needed for the proof of Corollary \ref{cor}.
\begin{lemma}\label{pqe}
An equivalent criterion for the boundedness of the partial quotients of an irrational $\a\in (0,1) $  is that $|\alpha_{k,1}|$ from the Minkowski chain for $\a$ is bounded away from zero.
\end{lemma}
\begin{proof}
If $M=L^b=\pmatrix10b1$ then
\begin{equation*}
	\alpha_{k,1} = \frac{q_k\alpha-p_k}{q_k'\alpha-p_k'} = \frac{q_{k-b}\alpha-p_{k-b}}{(q_{k-b}\alpha-p_{k-b})+b(q_{k-b}'\alpha-p_{k-b}')} = \frac{1}{1+b\alpha_{k-b,1}}.
\end{equation*}
The other three cases are similar.
In each case we see that $|\alpha_{k,1}|$ is bounded below if and only if $\sup\{a_j\}$ is finite.
\end{proof}

\begin{remark}
	The original paper by Hurwitz \cite{Hur} is still a good reference for the Hurwitz chain.
A modern reference is  \cite{Phi}, which also details  its relation to semi-regular continued fractions.  The dynamical properties
of the Hurwitz chain are discussed in \cite{Lag}, where it is called the additive continued fraction.
\end{remark}

\section{Successive minima}\label{mp}

In this section we will give what is essentially  Minkowski's proof of Theorem \ref{mac}.
A crucial ingredient is his theorem on successive minima in 
the geometry of numbers. 

For a general norm $F$ on $\R^\ell$ and any full lattice $\L\subset \R^\ell$  let 
\[
\mu_1\leq \mu_2\leq \cdots\leq \mu_\ell
\]
be the successive minima of $\L$ with respect to $F$. 
This means that  $\mu_j$ is the infimum over all $\mu>0$ such that there are $j$ linearly independent points $v\in \L$ with $F(v)\leq \mu.$
There exist (not necessarily unique) minimizing vectors 
$w_1,\dots,w_\ell\in \L$, which means that they are  linearly independent 
and satisfy $F(w_j)=\mu_j$ for $j=1,\dots,\ell.$ Note that $\{w_1,\dots,w_\ell\}$ do not  necessarily form a $\Z$-basis for $\L$.

The following fundamental result was first proved in \cite[Kap.~V]{Mink2}.
Shorter proofs were given by Davenport \cite{Dav1} and Weyl \cite{We}.  See also \cite{Cas}.
\begin{theorem*}[Minkowski's Theorem on Successive Minima]
Suppose that $\L$ has determinant one. Then
\[
\mathrm{vol}(\B) \mu_1\cdots \mu_\ell\leq 2^\ell,
\]
where $\mathrm{vol}(\B)$ is the volume of $\B=\{x\in \R^\ell; F(x)<1\}$.
\end{theorem*}

We remark that for the proof of Theorem \ref{mac} we can get by with a weaker result that replaces $2^\ell$ by a larger constant. In fact,  Minkowski gives a proof of this result in his paper with the constant $2^\ell \ell!$ in place of $2^\ell.$ See also  \cite[Cor.~2B p.~88]{Sch3} for a proof with the constant $2^\ell  \ell^{\frac{\ell}{2}}$, which is based on the case of an ellipsoid and a theorem of Jordan.

For $r=(q_1,\dots,q_\ell)\in \Z^{\ell}$ define
\begin{equation}\label{xi}
\xi(r)\defeq q_\ell+L_\a(q_1,\dots,q_n).
\end{equation}
For $m\in \Z^+$ recall the integral $\ell \times \ell $ matrix $A_m=(a_{i,j})$  defined above. From (\ref{be}) we have for each $i=1,\dots,\ell$ that
 \begin{equation}\label{beta1}
 \b_i=\xi(a_{i,1},\dots,a_{i,\ell}).
 \end{equation}
\begin{lemma}\label{minu} Fix $m \in \Z^+$ and 
suppose that $r_1,\dots,r_\ell\in \Z^\ell$ are linearly independent and satisfy $\|r_i\|_\infty \leq m$ for $i=1,\dots,\ell.$
Let them be ordered so that 
\begin{equation}\label{l3e}
|\xi(r_1)|\leq |\xi(r_2)|\leq \cdots\leq |\xi(r_\ell)|.
\end{equation}
Then for each $i=1,\dots \ell$ we have that
\[
|\b_i| \leq |\xi(r_i)|.
\]
\end{lemma}
\begin{proof}
For a fixed $m \in \Z^+$ let $w_j=(a_{i,1},\dots a_{i,\ell})$ denote the $i^{th}$ row of $A_m$, which is the integral vector produced by the Minkowski algorithm.
Thus for $j=1,\dots, \ell$, we know that $|\b_j|$ gives the smallest value of $|\xi(w)|$ for any $w\in \Z^\ell$ with $\|w\|_\infty \leq m$  that is linearly independent of $\{w_1,\dots,w_{j-1}\}.$

Note that at least $\ell-1$ of the $\{r_1,\dots,r_\ell\}$ are independent of $w_1$ and so each of those $r_k$ satisfies $|\xi(r_k)|\geq |\b_2|.$
At least $\ell-2$ of the $r_k$ are independent of $\{w_1,w_2\}$ and so these $r_k$ satisfy $|\xi(r_k)|\geq |\b_3|.$
Continue this process until we have at least one $r_k$ that satisfies $|\xi(r_k)|\geq |\b_\ell|.$
By (\ref{l3e}) we know that this last set of $r's$ must contain $r_\ell$ and so  $|\xi(r_\ell)|\geq |\b_\ell|$. Working backward we can finish the proof.
\end{proof}

\begin{lemma}\label{min} Fix $m \in \Z^+$ and let $A_m=(a_{i,j})$ and $\b_1,\dots, \b_\ell$  be from the Minkowski algorithm.
Let $\L=\Z^\ell$ and  define the norm on $\R^\ell$ by
\begin{equation*}
G_m(x_1,\dots,x_\ell)=\max\big(|x_1|,\dots,|x_\ell|,\tfrac{m}{|\b_\ell|}|L_\a(x_1,\dots,x_n)+x_{\ell}|\big).
\end{equation*}
Let $\mu_1\leq \mu_2\leq\cdots\leq \mu_\ell$ be the successive minima of $G_m$. Then
\begin{equation}\label{mu1}
\mu_\ell\geq m.
\end{equation}
\end{lemma}
\begin{proof}
Note that for
$r=(q_1,\dots,q_n,p)\in \Z^\ell$ we have
\begin{equation}\label{normd}
G_m(r)=\max\big(|q_1|,\dots,|q_n|,|p|,\tfrac{|\xi(r)|m}{|\b_\ell|}\big)
\end{equation}
where $\xi(r)$ was defined in (\ref{xi}).
Suppose that
$\{r_1,\dots,r_\ell\}$ are independent and such that for each $j$ we have
$
G_m(r_j)=\mu_j.
$ By (\ref{normd}) we see that if $\|r_j\|_\infty >m $ for any $j=1,\dots,\ell$  then $\mu_\ell >m.$
Otherwise apply Lemma \ref{minu} to $\{r_1,\dots,r_\ell\}$ to conclude that
$|\xi(r_j)|\geq |\b_j|$ for each $j=1,\dots, \ell.$ Therefore in particular for $j=\ell$ we get by (\ref{normd}) again that $\mu_\ell=G_m(r_\ell)\geq m.$
Thus in any case we have (\ref{mu1}).
\end{proof}

 Minkowski only proved the following result for $L_{(\a,\dots,\a^n)}$ where $\a $ is algebraic of degree $\ell$,
but his proof extends naturally.

\begin{lemma}\label{baib}
Suppose that $\a=(\a_1,\dots,\a_n)\in \R^n.$
If $L_\a$ is badly approximable then there are constants $c,C>0$ depending only on $\a$ such that
\begin{equation}\label{inn}
cm^{-n}<|\b_1|<\cdots<|\b_\ell|<Cm^{-n}.
\end{equation}

\end{lemma}
\begin{proof}

In this proof and those that follow we usually name and keep track of constants that depend only on $\a$, even
though it would be cleaner to use the  $\ll$ or $\gg$ notation.
 We do this to help the reader  verify inequalities.

Fix $m \in \Z^+$ and let $A_m=(a_{i,j})$ and $\b_1,\dots, \b_\ell$  be from the Minkowski algorithm.
Note that we suppress in the notation the dependence of $\b_j$ on $m.$
Let now $\L=\Z^\ell$ and $G_m$ the norm on $\R^\ell$ in Lemma \ref{min}.
The form $L_\a$ being badly approximable means that there is a $c>0$ so that
\begin{equation}\label{ba3}
|\xi(r)|>c\|q\|_\infty^{-n}
\end{equation}
for all $r=(q_1,\dots,q_n,p)\in \Z^\ell.$
By the definition of $\b_1$ and (\ref{ba3}) we have that
\begin{equation}\label{lower}
|\b_1|=\min_{\|r\|_\infty\leq m}|\xi(r)|> \tfrac{c}{m^n}.
\end{equation}
Now $G_m(r_1)=\mu_1$ and so (\ref{normd}) implies that 
\begin{equation}\label{47}
|\xi(r_1)|\leq \tfrac{\mu_1|\b_\ell|}{m}
\end{equation}
and also that $\|r_1\|_\infty \leq \mu_1$.
Thus by (\ref{ba3}) again we also have that 
\begin{equation}\label{48}
|\xi(r_1)|\geq \tfrac{c}{\mu_1^n}.
\end{equation}
By (\ref{47}) and (\ref{48}) we conclude that 
\begin{equation}\label{50}
\Big(\frac{\mu_1^\ell|\b_\ell|}{m}\Big)^n\geq c^n,
\end{equation}
which is the form we will need.

A straightforward calculation shows that 
\[
\mathrm{vol}(\{x\in \R^\ell; G_m(x)<1\})\geq V\tfrac{|\b_\ell|}{m},
\]
where $V>0$ is a constant depending only on $\a$.
By Minkowski's theorem on successive minima we have
\[
V\tfrac{|\b_\ell|}{m}\mu_1^n\mu_\ell \leq 2^\ell
\]
so that using (\ref{50}) 
we have for $M=2^{\ell^2}V^{-\ell}$  that
\[
 c^n\tfrac{|\b_\ell|}{m}\mu_\ell^\ell\leq \Big(\tfrac{|\b_\ell|}{m}\mu_1^n\mu_\ell\Big)^\ell\leq M.
\]
Thus 
$ |\b_\ell|\leq \frac{Cm}{\mu_\ell^\ell}$ for $C=\frac{M}{c^n}.$
Finally, from Lemma \ref{min} we derive that $|\b_\ell|\leq \frac{C}{m^n}.$
Together with (\ref{lower}), this finishes the proof of Lemma \ref{baib}.
\end{proof}

\subsection*{Proof of Theorem \ref{mac}}

The proof of the implication {\it algebraic implies finite} works the same for the Minkowski chain for
 any $\Q$-basis  $\{\a_1,\a_2,\dots,\a_n,1\}$ of a real number field $K$.
Recall that for each $k$ we have \[(\b_1,\dots,\b_\ell)=B_k(\a_1,\dots,\a_n,1),\] where again we suppress the dependence of $\b_j$ on $k$ in the notation.
Clearly $\b_j\in K$  and there is a positive integer $b$ such that $b\b_j$ is an algebraic integer for any $j,k,$ so we must have $N_{K/\Q}(b\b_j)\geq1.$
 Denote by $\{\b_i^{(j)}; j=1,\dots, \ell\}$ the set of Galois conjugates of $\b_i=\b_i^{(1)}$.
Set $C_1=\max_{j=1,\dots,\ell}(1+|\a_1^{(j)}|+\cdots+|\a_n^{(j)}|)$.  Clearly 
\begin{equation}\label{n2}|
\b_i^{(j)}|\leq C_1 m,
\end{equation} 
where $m=m_k$ from the algorithm.

We know that $L_{(\a_1,\dots,\a_n)}$ is badly approximable so by Lemma \ref{baib} we have that for each $i$
\begin{equation}\label{bew}
|\b_i |\leq C m^{-n}.
\end{equation}
Therefore we have that
\begin{equation}\label{norm}
b^{-\ell}\leq |N_{K/\Q}(\b_i)|=|\prod_{j=1}^\ell \b_i^{(j)}|  \leq C\,C_1^n.
\end{equation}

From the first inequality of (\ref{norm}), (\ref{n2}) and (\ref{bew}) we get that for $k>1$
\begin{equation}\label{bew2}
|\b_i^{(k)}|\geq C_2 m
\end{equation}
for some constant $C_2>0$ depending only on $\a$.
Here we have used (\ref{bew}) for the first factor in the product and  (\ref{n2}) for all of the  remaining factors except for the $k^{th}$.
Recall from (\ref{alph}) that
\[
(\a_{k,1},\dots ,\a_{k,n})=(\tfrac{\b_1}{\b_\ell},\dots, \tfrac{\b_n}{\b_\ell})\in K^n.
\]
Let $\g_{k,i}=b^\ell N_{K/\Q}(\b_\ell)\a_{k,i}$. Then $\g_{k,i}$ is an algebraic integer in $K$.
From (\ref{inn}),  (\ref{n2}), (\ref{norm}) and (\ref{bew2})  we have  for each $i,j,k$ that $|\g_{k,i}^{(j)}|\leq C_3$ for some $C_3>0$ that depends only on $\a$.
It follows that there are only finitely many such $(\g_{k,1},\dots ,\g_{k,n})$ hence only finitely many
$(\a_{k,1},\dots ,\a_{k,n}).$

For the converse, we need to  assume that $(\a_1,\a_2,\dots,\a_n)=(\a^n,\dots,\a)$ and suppose that there are only finitely many values  of the sequence $\{B_k(\a^n,\dots,\a)\}_{k\geq 1}$.
Then for some $k'>k$  we have $B_k(\a^n,\dots,\a)=B_{k'}(\a^n,\dots,\a)$.
Hence
\begin{equation*}
B_k^{-1}B_{k'}(\a^n,\dots,\a)=(\a^n,\dots,\a),
\end{equation*}
which implies that
 \begin{equation}\label{row}
 B_k^{-1}B_{k'}(\a^n,\dots,\a,1)^\top=\theta (\a^n,\dots,\a,1)^\top\end{equation}
for some $\theta\in \R$ with $|\theta|\leq 1.$
Write $B_k^{-1}B_{k'}=(c_{i,j})$ so that for each $k=1,\dots,n$ we can write two successive rows of (\ref{row}) as
\begin{align*}
c_{k,1}\a^n+\cdots + (c_{k,k}-\theta)\a^{n-k+1}+\cdots+c_{k,n}\a+c_{k,\ell}&=0\\
c_{k+1,1}\a^n +\cdots+ (c_{k+1,k+1}-\theta) \a^{n-k}+\cdots+c_{k+1,n}\a+c_{k+1,\ell}&=0.
\end{align*}
Multiply the first equation by $\a$ and subtract rows to get for each $k=1,\dots,n$ that
\[
c_{k,1}\a^\ell+(c_{k,2}-c_{k+1,1})\a^n+\cdots+(c_{k,k}-c_{k+1,k+1})\a^{n-k}+\cdots+(c_{k,\ell}-c_{k+1,n})\a- c_{k+1,\ell}=0.
\]
Unless all of these vanish identically we see that $\a$ is algebraic of degree $\ell$, upon
using that we are assuming that $\{\a^n,\dots,\a,1\}$ are linearly independent over $\Q$.
If they all vanish it follows easily that $c_{i,j}=\d_{i,j}c$ for some $c\in \Q$.
Thus $c=\theta$ and $B_{k'}=\theta B_{k}$ with $|\theta|<1$, which contradicts that $k'>k.$
\qed

\section{Reduced bases}\label{redu}

In this section we present several well-known results from the theory of reduced bases that we need.
Perhaps the best reference for this material   is a set of unpublished notes from a seminar given at IAS 
in 1949 \cite{WSM}.  Because these notes  might not be readily available we have included  here all proofs.
Another reference is \cite{GL}.

Let $\L\subset \R^\ell$ be a full lattice and $F$ a norm on $\R^\ell.$
The lattice points in $\L$ taking on the successive minima on $F$ are linearly independent but do not necessarily form a basis 
for $\L$. Minkowski's second theorem implies, with some extra work, a substitute that bounds the 
product of values of $F$ of the elements of a reduced basis.
In case $F$ is a positive definite quadratic form the theory was developed by Minkowski \cite{Mink1.6}.

Suppose that $\{v_1,\dots,v_\ell\}$ is an ordered  $\Z$-basis for $\L$.
Define  for $k=1,\dots, \ell$
\begin{equation}\label{Rk}
R_k=\{a_1v_1+\cdots +a_\ell v_\ell; \;a_j\in \Z\;\; \text{with}\;\;\gcd (a_k,a_{k+1},\dots ,a_\ell)=1\} \subset\L.
\end{equation}
Note that  $v_j\notin R_k$ for $j<k.$

In general, an (ordered)  $\Z$-basis $\{v_1,\dots,v_\ell\}$ for $\L$ is {\it reduced} with respect to $F$ if for each $k=1,\dots ,\ell$ we have that for all $v\in R_k$ 
\[
F(v)\geq F(v_k).
\]
It follows that if $\{v_1,\dots,v_\ell\}$  is reduced with respect to $F$ and $\l_k=F(v_k)$ then
\[
\l_1\leq \l_2\leq \cdots\leq \l_\ell.
\]

\begin{lemma}\label{rede}
For any norm $F$ and full lattice $\L\subset \R^\ell$ reduced bases $\{v_1,\dots,v_\ell\}$ exist.
\end{lemma}
\begin{proof}
The beginning of the proof is similar to the construction of the $A$ in Minkowski's algorithm except that now we demand that 
$A\in \GL(\ell,\Z)$. 
Let $\{u_1,\dots,u_\ell\}$ be any $\Z$-basis for $\L$.
Let $A=(r_1,\dots,r_\ell)\in \GL(\ell,\Z)$ where $r_i$ is a column vector.
Choose $r_1$ so that 
\[
\l_1=F(v_1)=F\big((u_1,\dots,u_\ell)r_1\big)
\]
is minimal. Note that this exists by convexity. Now choose $a_2$  to minimize $\l_2=F(v_2)=F((u_1,\dots,u_\ell) r_2).$  Thus $\l_1\leq \l_2.$ Continue this process to determine $A$ and the basis $\{v_1,\dots,v_\ell\}$ where for $\l_k=F(v_k)$ we have that $\l_1\leq \l_2\leq \cdots\leq \l_\ell.$
We want to show that $\{v_1,\dots,v_\ell\}$ is  reduced.

Fix $k$ with $1\leq k\leq \ell.$
By construction if $s$ is any column of a matrix in $\GL(\ell,\Z)$ that is linearly independent of $\{r_1,\dots,r_{k-1}\}$
then
\begin{equation}\label{li}
F\big((u_1,\dots,u_\ell)s\big)\geq F\big((u_1,\dots,u_\ell)r_k\big).
\end{equation}
Let $q^\top=(q_1,\dots,q_\ell)\in \Z^\ell$ be any integral vector such that $\gcd(q_k,\dots,q_{\ell})=1.$
Fix a matrix of the form
\[
A'=\pMatrix{I}{B}{0}{C}\in \GL(\ell,\Z)
\]
where $I$ is the $(k-1)\times (k-1)$ identity matrix and where the $k^{th}$ column of $A'$ is $q$. This is possible by our assumption on $q$. 
Clearly the first $k-1$ columns of $AA' $ coincide with those of $A$.
Hence if $s$ is the $k^{th}$ column of $AA'$ then by (\ref{li})
\[
F(v_k)=F\big((u_1,\dots,u_\ell)r_k\big)\leq F\big((u_1,\dots,u_\ell)s\big)=F\big((u_1,\dots,u_\ell)Aq_k\big)=F\big((v_1,\dots,v_\ell)q_k\big).
\]
It follows that $\{v_1,\dots,v_\ell\}$ is reduced.
\end{proof}
The statement of Part (i) of the following lemma is given in  \cite[Lemma 2 p.~100]{Sie} with a different proof
than the one we give below.
Our proof is adapted from the proof of Part (ii) given in \cite{WSM}.
Part (i) is crucial for
 our proofs of Theorems \ref{t3}, \ref{t4} and \ref{t1}.

\begin{lemma}\label{lred}
Let $F:\R^\ell\rightarrow [0,\infty)$ be a norm and $\L\subset \R^\ell$ be a full lattice.
Suppose that $v_1,\dots,v_{\ell}$ is a reduced basis for $\L$ with respect to $F$
 so that for $\l_i=F(v_i)$ 
\[
\l_1\leq \l_2\leq \cdots \leq \l_\ell.
\]
\begin{enumerate}[label=(\roman*)]
\item If
$u_1,\dots,u_\ell\in \L$ is any linearly independent set in $\L$  ordered so that for $\nu_j=F(u_j)$
\[
\nu_1\leq \nu_2\leq \cdots \leq \nu_\ell
\]
then for each $k=1,\dots, \ell$ we have that
$
\l_k\leq (\tfrac{3}{2})^{k-1}\nu_k.
$
\item If
$w_1,\dots,w_\ell\in \L$ are minimizing vectors in  $\L$  with successive minima $\mu_j=F(w_j)$
\[
\mu_1\leq \mu_2\leq \cdots \leq \mu_\ell
\]
then $\l_1=\mu_1$ and  for each $k=2,\dots, \ell$ we have that
$
\l_k\leq (\tfrac{3}{2})^{k-2}\mu_k.
$
\end{enumerate}
\end{lemma}

\begin{proof}
(i).
 There are $a_{i,j}\in \Z$ such that for each $i$
\[
u_i=\sum_{1\leq j \leq \ell}a_{i,j}v_j.
\]
Fix $k$ with $1\leq k \leq \ell.$
Since $\{u_1,\dots,u_\ell\}$ are linearly independent there is a $j$ with $1\leq j\leq k$ so that
\[
a_{j,k}, a_{j,k+1},\dots, a_{j,\ell}
\]
are not all zero. Thus for any such $j$ let $d=\gcd(a_{j,k}, a_{j,k+1},\dots, a_{j,\ell})>0.$

If $d=1$ then $u_j\in R_k$ and hence 
\[
\l_k\leq F(u_j)=\nu_j\leq \nu_k.
\]
If $d>1$ 
define for $m=1,\dots,k-1$ the integer $r_m$ with $|r_m|\leq \tfrac{d}{2}$ so that
\[
a_{j,m}+r_m\equiv 0\pmod{d}.
\]
Then
$
y=\tfrac{1}{d}(v_j+r_1 v_1+\cdots+r_{k-1}v_{k-1})\in R_k.
$
Hence we have 
\begin{align*}
\l_k\leq& F(y)\leq \tfrac{1}{d}F(u_j)+\tfrac{r_1}{d}\l_1+\cdots +\tfrac{r_{k-1}}{d} \l_{k-1}\\
\leq& \tfrac{\nu_k}{2}+\tfrac{1}{2}(\l_1+\cdots+\l_{k-1}).
\end{align*}
Therefore in any case for $k=1,\dots \ell$ we have
\begin{align}\label{sie}
\l_k\leq \nu_k+\tfrac{1}{2}(\l_1+\cdots+\l_{k-1}).
\end{align}
Suppose now that for $j=1,\dots, k-1$ 
\[
\l_j\leq (\tfrac{3}{2})^{j-1}\nu_j.
\]
Then by (\ref{sie}) we deduce that
\[
\l_k\leq \nu_k+\half \big((\tfrac{3}{2})^0+(\tfrac{3}{2})^1+\cdots+(\tfrac{3}{2})^{k-2}\big)\nu_k =(\tfrac{3}{2})^{k-1}\nu_k.
\]
Since $\l_1\leq \nu_1$ the result (i)  follows by induction.

The proof of (ii) is similar except that we use the fact that $\l_1=\mu_1$.
\end{proof}

The following result was found independently by Mahler \cite{Mah1} and Weyl \cite{We}.
\begin{theorem}[First Finiteness Theorem]\label{fft}
Let $\L\subset \R^\ell$ be a full lattice with determinant 1. 
  For a reduced basis $\{v_1,\dots,v_\ell\}$ with $\l_k=F(v_k)$ we have
\begin{equation}\label{ffi}
\tfrac{2^\ell}{\ell!}\leq \mathrm{vol}(\B)\,\l_1\cdots\l_\ell\leq 2^\ell (\tfrac{3}{2})^{\frac{(\ell-1)(\ell-2)}{2}},
\end{equation}
where $\mathrm{vol}(\B)$ is the volume of $\B=\{x\in \R^\ell; F(x)<1\}$.
\end{theorem}
\begin{proof}
The first inequality is a consequence of the fact that the closure of $\B$ contains the octahedron with vertices at the points 
\[
\{\pm\tfrac{v_1}{\l_1},\dots, \pm\tfrac{v_\ell}{\l_\ell}\}
\]
and this octahedron has volume $\frac{2^\ell}{\ell! \l_1\cdots\l_\ell},$ which is easily found by computing the determinant of the linear transformation that maps the $k^{th}$ standard unit vector to $\tfrac{v_k}{\l_k}$ for each $k.$

The second inequality is an immediate consequence of (ii) of Lemma \ref{lred}  and Minkowski's Second Theorem.
\end{proof}
We remark that we could also apply (i) of Lemma \ref{lred} to get the second inequality in (\ref{ffi}) with the right hand side multiplied by $\frac{3}{2}$, which would be sufficient for our purposes.

\section{Criteria for badly approximable and singular forms}\label{pf}
In this section we will prove Theorems \ref{t3}, \ref{liou}  and \ref{t4}.
We make use of the lattice $\L_t(\a)\subset \R^{\ell}$
 defined in terms of $\a$  for a fixed parameter $t>0$ by 
\begin{equation}\label{lamb}
\L_t=\L_t(\a)=(t^{-1},0,\dots,0,\a_1t^n)\Z+\cdots+(0,0,\dots,t^{-1},\a_nt^n)\Z+(0,0,\dots,0,t^n)\Z.
\end{equation}
Clearly $\det(\L_t)=1.$
Consider the norm on $\R^\ell$ given by
\begin{equation}\label{sup}
F_\infty(x_1,\dots,x_n,y)=\|( x_1,\dots,x_n,y)\|_\infty.
\end{equation}

The next lemma follows as  a special case from results of \cite{Da}.
For convenience we give the proof here, which for our case is quite simple.
\begin{lemma}\label{l1}
For the lattice $\L_t(\a)$ defined above let \[\l_1(t)=\min_{\substack{v\in \L_t(\a)\\v\neq 0}}F_\infty(v).\]
 \begin{enumerate}[label=(\roman*)]
\item The form $L_\a$ is badly approximable if and only if there is a $c>0$ depending only on $\a$ such that
$\l_1(t)>c$
for all $t\geq 1$.
\item The form $L_\a$ is singular if and only if 
$\l_1(t)\rightarrow 0$ as $t \rightarrow \infty.$
\end{enumerate}
\end{lemma}
\begin{proof}

Part (i):
First suppose that $\a$ is badly approximable, so that (\ref{ba}) holds with some $c_\a>0$.
Fix $t\geq 1$ and $v=(x_1,\dots,x_n,y)\in \L_t(\a)$. If $x_1,\dots, x_n=0$ then $yt^{-n} $ is a non-zero integer so $F_\infty(v)\geq 1.$
Thus suppose that $x=(x_1,\dots,x_n)\neq 0$, set $c=c_\a^{\frac{1}{\ell}}$ and $\mu=\|x\|_\infty>0.$

If $\mu\leq c$ then
$
t^n\|L_\a(tx)\|>c,
$
so that $F_\infty(v)>c.$
If $\mu>c$ then  again
$F_\infty(v)>c.$ 
In either case it follows that $\l_1(t)>c$.

For the converse assertion, suppose that $\l_1(t)>c$.  Fix non-zero  $q=(q_1,\dots,q_n)\in \Z^n$.
Assume that $c<1.$ Next choose $t=c^{-1}\|q\|_\infty$ so that $t>1$ and $t^{-1}\|q\|_\infty=c.$
For any integer $p$ we have
$(t^{-1}q_1,\dots,t^{-1}q_n,t^n(p+L_\a(q)))\in \L_t(\a)$
 hence
$
t^n|L_\a(q)+p|>c,
$
which implies that \[\|L_\a(q)\|>c^{\ell}\|q\|^{-n}_\infty,\] so $L_\a$ is badly approximable.

\bigskip\noindent
Part (ii):
Suppose that $L_\a$ is singular and $\ep>0$ is fixed.    For sufficiently large $t$ 
there exists $q\in \Z^n$ 
and $p\in \Z$ such that  $\|q\|_\infty\leq t \ep^{\frac{1}{\ell}}$ and
\[
|p+L_\a(q)|\leq t^{-n}\ep^{1-\frac{n}{\ell}}= t^{-n}\ep^{\frac{1}{\ell}}.
\]
Let
$v=(t^{-1}q_1,\dots, t^{-1}q_n,t^n(p+L_\a(q)))\in \L_t(\a);$
 for sufficiently large $t$ we have $F_\infty(v)\leq \ep^{\frac{1}{\ell}},$ proving that $\l_1(t)\rightarrow 0$ as $t \rightarrow \infty.$

The converse is similar and is left to the reader.
\end{proof}

\subsection*{Proof of  Theorem \ref{t3} }

We have shown in Lemma \ref{baib} that if $L_\a$ is badly approximable then 
\[
\frac{|\b_1|}{|\b_\ell|} \geq \frac {c}{C}.
\]

Now suppose that $L_\a$ is not badly approximable. By (i) of Lemma \ref{l1}, for any $\ep>0$ 
there exists some $t\geq 1$ and $v \in \L_t(\a)$ so that $F(v)<\ep$
where again 
\[
F(v)=F_\infty(v)=\|v\|_\infty.
\]
Let $\{v_1,v_2,\dots,v_\ell\}$ be a reduced basis for $\L_t(\a)$ with respect to $F$
and such that for $\l_i=F(v_i)$ we have
$\l_1\leq \cdots \leq \l_\ell.$

Now $v_1=(t^{-1}q_{1},\dots,t^{-1}q_{n}, t^n\xi(r))$ for some non-zero $r=(q_{1},\dots,q_{n},p)\in \Z^\ell, $ where $\xi(r)$ was defined in (\ref{xi}).
Clearly we have
\begin{equation}\label{Fv1}
\l_1=F(v_1)\leq F(v)<\ep
\end{equation}
so from the definition of $F$
\begin{equation}\label{Ft}
t^{-1}|q_j|<\ep \;\;\;\text{for $j=1,\dots,n$ and}\;\;\; t^n|\xi(r)| <\ep.
\end{equation}

Next set $m = \lceil \kappa t\ep \rceil$, where $\kappa$ is a constant depending only on $\a$ chosen to be large enough so that $\max(|q_1|,\dots,|q_n|,|p|) \leq m$, which is possible by (\ref{Ft}).
For $A_m=(a_{i,j})$ from Minkowski's algorithm let for each $i=1,\dots, \ell$
\begin{equation}\label{nor}
u_i=(t^{-1}a_{i,1},\dots,t^{-1}a_{i,n}, t^n\b_i)\in \L_t(\a).
\end{equation}
By the definition of $\b_1$ and (\ref{Ft}) again we have that
\begin{equation}\label{in1}
|\b_1|\leq |\xi(r)|<t^{-n}\ep.
\end{equation}

By construction
\begin{equation}\label{beta}
t^n|\b_1|<t^n|\b_2|<\cdots<t^n|\b_\ell|,
\end{equation}
but we do not know that necessarily 
\[
F(u_1)\leq F(u_2)\leq \cdots\leq F(u_\ell).
\]
Let $k\in \{1,\dots,\ell\}$ be such that $F(u_k)=\max( F(u_1),\dots,F(u_\ell)).$
Since $\{u_1,\dots,u_\ell\}$ are linearly independent in $\L_t(\a)$,
by (i) of Lemma \ref{lred} we have that 
 \begin{equation}\label{Fv2}
 F(u_k) \geq (\tfrac{2}{3})^{n}F(v_\ell).
 \end{equation}
 By the first inequality of the  First Finiteness Theorem and (\ref{Fv1}) we see that
$   F(v_\ell) >(\tfrac{1}{\ell ! \ep})^{1/n} $
 and therefore by (\ref{Fv2}) 
 \begin{equation}\label{Fv3}
  F(u_k)  > (\tfrac{2}{3})^n(\tfrac{1}{\ell ! \ep})^{1/n}.
\end{equation}
Now \[\max(t^{-1}|a_{k,1}|,\dots,t^{-1}|a_{k,n}|)\leq \tfrac{m}{t}\leq \kappa \ep+t^{-1}<(\tfrac{2}{3})^n(\tfrac{1}{\ell ! \ep})^{1/n}\]
for $\ep>0$ sufficiently small.
Hence by this, (\ref{Fv3})  and (\ref{beta}) we have
\begin{equation} \label{eq:b-ell-lower}
|\b_\ell| \geq|\b_k|> t^{-n}(\tfrac{2}{3})^n(\tfrac{1}{\ell ! \ep})^{1/n},
\end{equation}
after referring again to (\ref{nor}).
By (\ref{in1}) we conclude that for sufficiently small $\ep$ 
\[
|\tfrac{\b_1}{\b_\ell}|< (\tfrac{3}{2})^n (\ell ! \ep)^{1/n}\ep.
\]
It follows that if $L_\a$ is not badly approximable then $|\tfrac{\b_1}{\b_\ell}|$ can be made arbitrarily small
for some $m$, hence $|\a_{k,1}|$ can be made arbitrarily small for some $k$.
\qed

\subsection*{Proof of Theorem \ref{liou}}

Our proof is adapted from an argument given in \S 3 of \cite{DS1}. 
By the definition of a Liouville number, there are infinitely many $q\in \Z^+$ and $p\in \Z$ prime to $q$ such that
\[
|\a-\tfrac{p}{q}|<q^{-n-2}.
\]
For $x\in \Z^n$ with $0<\|x\|_{\infty}\leq Q=q-1$ we have
\begin{align*}
L_{(\a^n,\dots,\a)}(x)-y=&x_1(\tfrac{p}{q})^n+x_2(\tfrac{p}{q})^{n-1}+\cdots+ x_n\tfrac{p}{q}-y+O(q^{-n-1})\\
=&\frac{x_1p^n+x_2 p^{n-1}q+\cdots+ x_npq^{n-1}-yq^n}{q^n}+O(q^{-n-1}).
\end{align*}
Now $x_1p^n+x_2 p^{n-1}q+\cdots+ x_npq^{n-1}-yq^n$ is a non-zero integer since $q\nmid x_1$.   Thus
for any $\ep>0$ there are arbitrarily large $Q$ so that
\[
\|L_{(\a^n,\dots,\a)}(x)\|\geq (1-\ep)Q^{-n}
\]
for all $x\in \Z^n$ with $0<\|x\|_{\infty}\leq Q$.
It follows that Dirichlet's theorem cannot be improved for $L_{(\a^n,\dots,\a)}$, so by the theorem of Davenport and Schmidt,
$L_{(\a^n,\dots,\a)}$ cannot be badly approximable. \qed

\subsection*{Proof of  Theorem \ref{t4} }

 Suppose that $\alpha$ is singular and let $\ep\in (0,1)$.
Then there exists a $t_0$ such that 
\begin{equation*}
	\lambda_1(t)<\ep \quad \text{ for all } t\geq t_0.
\end{equation*}
This is the analogue of \eqref{Fv1} but now the inequality holds for all sufficiently large $t$.
Let $m$ be any positive integer greater than $t_0$ and let $t=m/\ep \geq t_0$.
If $A_m=(a_{i,j})$ is the $m$-th matrix from Minkowski's algorithm, then by following the argument between \eqref{nor} and \eqref{eq:b-ell-lower} above we find that 
\begin{equation*}
	|\beta_1|<\frac{\ep}{t^n} \quad \text{and} \quad |\beta_\ell|>\frac{c}{t^{n}\ep^{1/n}}
\end{equation*}
for some $c>0$.
It follows that $\left|\frac{\beta_1}{\beta_\ell}\right|$ can be made arbitrarily small for all sufficiently large $m$, hence $|\alpha_{k,1}|$ can be made arbitrarily small for all sufficiently large $k$.

Conversely, suppose that $\alpha$ is not singular.
Then there exists a $c>0$ and a sequence $\{Q_j\}$ tending to infinity such that for each $j$ there are infinitely many $q\in \Z^n$ with 
\begin{equation*}
	\|q\|_{\infty}\leq Q_j \quad \text{ and } \|L_\alpha(q)\|\geq cQ_j^{-n}.
\end{equation*}
Fix one of these $Q_j$ and let $m=Q_j$.
Then
\begin{equation*}
	\beta_1 = \min_{\substack{q\in \Z^n\setminus\{0\} \\ \|q\|_\infty \leq m}} |\xi(q)| \geq \frac{c}{m^n}.
\end{equation*}
This is analogous to \eqref{lower} but now the lower bound only holds for a sequence of $m$ tending to infinity.
Note that \eqref{mu1} is true for these $m$, and following \eqref{47}--\eqref{50} we find that $|\b_\ell|\leq \frac{C}{m}$ holds here as well.
It follows that for infinitely many $m$ we have $\left|\frac{\beta_1}{\beta_\ell}\right| \geq c'$ for some $c'>0$ and thus $|\alpha_{k,1}|$ is bounded away from zero for infinitely many $k$. \qed

\section{A variant of Dirichlet's theorem}\label{di}
\subsection*{Proof of  Theorem \ref{t1}}
The proof of Theorem \ref{t1}  makes use of reduced bases of the lattice $\L_t(\a)\subset \R^{\ell}$
defined in (\ref{lamb}) with respect to the sup-norm $F=F_\infty$ defined in (\ref{sup}).

Suppose that $\a$ is badly approximable and 
fix $t\geq 1$. 
 There exists a reduced basis $\{v_1,\dots,v_\ell\}$ of $\L_t(\a)$ with respect to $F$,
where $v_i=(x_{i,1},\dots,x_{i,n},y_i)$ for each $i=1,\dots,\ell.$ Write
$\l_i=F(v_i)$  ordered so that
\[
\l_1\leq \l_2\leq \cdots \leq \l_\ell.
\]
By the First Finiteness Theorem (Theorem \ref{fft})
\begin{equation}\label{first}
\tfrac{1}{\ell!}\leq F(v_1)\cdots F(v_\ell)\leq (\tfrac{3}{2})^{\frac{(\ell-1)(\ell-2)}{2}}.
\end{equation}
By Lemma \ref{l1} (i) we have $F(v_i)> c>0$ for each $i$ so that there is a $c_2>0$ with $F(v_i)<c_2$ for each $i$.
Thus there exists $A=(a_{i,j})\in \GL(\ell,\Z)$ such that $|a_{i,j}|< c_2t$ for $i=1,\dots,\ell$ and $j=1,\dots,n$
and such that 
\begin{equation}\label{Ai}
\| A(\a_1,\dots,\a_n,1)^\top\|_\infty < c_2 \,t^{-n}.
\end{equation}
There exists $c_3>0$ depending on $\a$ so that (\ref{Ai}) holds 
for some $A=(a_{i,j})\in \GL(\ell,\Z)$ with $\|A\|_\infty < c_3 t$.
Taking $t=c_3^{-1}Q$ we deduce the first part of Theorem \ref{t1}.

Conversely, suppose that $L_\a$ is not badly approximable. By Lemma \ref{l1} (i) again, for any $\ep>0$ there is a 
$t\geq 1$ and a $v=(x_1,\dots,x_n, y)\in \L_{t}(\a)$ such that $F(v)<\ep.$
Let $\{v_1,\dots,v_\ell\}$ be a reduced basis for $\L_{t}(\a)$ with respect to $F$.
Suppose that
$u_1,\dots,u_\ell\in \L_{t}$ is a basis for $\L_{t}(\a)$  ordered so that for $\nu_j=F(u_j)$ we have
\[
\nu_1\leq \nu_2\leq \cdots \leq \nu_\ell.
\]
By (i) of Lemma~\ref{lred} 
 for each $k=1,\dots, \ell$ we have that
\[
\l_k\leq (\tfrac{3}{2})^{k-1}\nu_k.
\]
Now 
$\ep >F(v)\geq F(v_1),$
 so by the first inequality of (\ref{first}) we must have that 
$F(v_\ell)\geq (\frac{1}{\ell! \ep})^{1/n}$.
Thus
\[
F(u_\ell)\geq (\tfrac{2}{3})^n(\tfrac{1}{\ell ! \ep})^{1/n}.
\]
Hence given any $c>0$, by choosing $\ep$ small enough, we can find a $Q$ 
such that
\[
\| A(\a_1,\dots,\a_n,1)^\top\|_\infty > c \,Q^{-n}
\]
for all $A\in \GL(\ell,\Z)$ with $\|A\|_\infty <Q.$
This completes the proof of Theorem \ref{t1}.
\qed


\begin{thebibliography}{1}

\bibitem{Cas}	 Cassels, J. W. S., An introduction to the geometry of numbers. Die Grundlehren der mathematischen Wissenschaften in Einzeldarstellungen mit besonderer Berücksichtigung der Anwendungsgebiete, Bd. 99 Springer-Verlag, Berlin-Göttingen-Heidelberg 1959 viii+344 pp.
  
\bibitem{Ch}	Cheung, Y., Hausdorff dimension of the set of singular pairs. {\it  Ann. of Math.} (2) 173 (2011), 127–167.

\bibitem{CC} Cheung, Y. \&  Chevallier, N.,  Hausdorff dimension of singular vectors. {\it Duke Math. J.} 165 (2016),  2273–2329. 

\bibitem{Da}  Dani, S. G.,  Divergent trajectories of flows on homogeneous spaces and Diophantine approximation. {\it J. Reine Angew. Math.} 359 (1985), 55–89.

\bibitem{DFSU1} 	Das, T., Fishman, L., Simmons, D. \& Urbański, M., A variational principle in the parametric geometry of numbers, with applications to metric Diophantine approximation.{\it C. R. Math. Acad. Sci. Paris} 355 (2017), 835–846.

\bibitem{Dav1} Davenport, H.,   Minkowski's inequality for the minima associated with a convex body.
{\it  Quarterly J. of Math.}, Volume os-10, Issue 1, (1939),  119–121.

\bibitem{DS1}  Davenport, H. \& Schmidt, W. M.,  Dirichlet's theorem on {D}iophantine approximation. 1970 Symposia Mathematica, Vol. IV (INDAM, Rome, 1968/69) pp. 113–132 Academic Press, London. 

\bibitem{GL}	Gruber, P. M. \& Lekkerkerker, C. G.,  Geometry of numbers. Second edition. North-Holland Mathematical Library, 37. North-Holland Publishing Co., Amsterdam, 1987. xvi+732 pp. 

\bibitem{Haj} Hajós, G.,  Über einfache und mehrfache Bedeckung des n-dimensionalen Raumes mit einem Würfelgitter. {\it Math. Z.} 47 (1941), 427–467.

\bibitem{Han}	Hancock, H.,  Development of the Minkowski geometry of numbers. Vols. One, Two. Dover Publications, Inc., New York 1964 Vol. One: xix+pp. 1–452. Vol. Two: ix+pp. 453–839.

\bibitem{Hur}  Hurwitz, A., Über die angenäherte Darstellung der Zahlen durch rationale Brüche, Math. Ann., 14, 1894, 417–436; and Oeuvres, tome II, pp. 137–156.

\bibitem{Jac}  Jacobi, C. G. J., Allgemeine Theorie der kettenbruchähnlichen Algorithmen, in welche jede Zahl aus drei vorhergehenden gebildet wird, {\it  J. Reine Angew. Math.} 69 (1868) 29–64. in Ges. Werke, Vol.VI, 385–426, Berlin Academy, (1891).

\bibitem{Khi}  Khintchine, A. Y., Zur metrischen Theorie der diophantischen Approximationen, {\it Math.Z.} 24 (1926), 706–714.

\bibitem{Khi2}  Khintchine, A. Y., \"Uber eine Klasse linearer diophantischer Approximationen, {\it Rendiconti del Circolo Matematico di Palermo} 50, (1926), 170-195.

\bibitem{Lag} 	 Lagarias, J. C.,  Number theory and dynamical systems. The unreasonable effectiveness of number theory (Orono, ME, 1991), 35–72, Proc. Sympos. Appl. Math., 46, Amer. Math. Soc., Providence, RI, 1992.


\bibitem{Mah1} 	Mahler, K.,	On Minkowski's theory of reduction of positive definite quadratic forms,
{\it Quart. J. Math. Oxford} 9 (1938), 259-262.

\bibitem{Mink1.2} Minkowski, H.,   Ein Kriterium f\"ur die algebraischen Zahlen.
{\it G\"ott. Nachr.} 
(1899) 64-88. in Gesammelte Abhandlungen Bd. I. 

\bibitem{Mink1.6} Minkowski, H.,  Diskontinuitatsbereich fur arithmetische Aquivalenz, {\it J. Reine Angew. Math.}, 129, (1905) 220–274  ,  Gesammelte Abhandlungen Bd II, (1911)  53–100.

\bibitem{Mink2} Minkowski, H.,  Geometrie der Zahlen, Teubner (1910).


\bibitem{Phi} Philippon, P.,  A Farey tail. {\it Notices Amer. Math. Soc.} 59 (2012), no. 6, 746–757.

\bibitem{Sch1}  Schmidt, W. M.,  Badly approximable systems of linear forms. {\it J. Number Theory} 1 (1969) 139–154.

\bibitem{Sch3}   Schmidt, W. M.,  Diophantine approximation. {\it Lecture Notes in Mathematics}, 785. Springer, Berlin, 1980. x+299 pp. 

\bibitem{Sie} Siegel, C. L.,  Lectures on the geometry of numbers. Notes by B. Friedman. Rewritten by Komaravolu Chandrasekharan with the assistance of Rudolf Suter.  Springer-Verlag, Berlin, 1989. x+160 pp.

\bibitem{We}  Weyl, H., On geometry of numbers. {\it Proc. London Math. Soc.}(2) 47, (1942). 268–289. 
\bibitem{WSM} Weyl, H.,  Siegel C.L., \&  Mahler, K.,  Seminar on Geometry of Numbers, IAS (1949).

\end{thebibliography}
\end{document}